\documentclass[a4paper,11pt]{amsart}
\usepackage[utf8]{inputenc}
\usepackage{verbatim}
\usepackage{amssymb}
\usepackage{amsmath, amsthm, amsfonts}
\usepackage{color} 

\newtheorem{theorem}{Theorem}[section]
\newtheorem{definition}[theorem]{Definition}
\newtheorem{prop}[theorem]{Proposition}
\newtheorem{lemma}[theorem]{Lemma}
\newtheorem{corollary}[theorem]{Corollary}
\newtheorem{remark}{Remark}

\newcommand{\vecz}{\mathbf{z}}
\newcommand{\vecw}{\mathbf{w}}
\newcommand{\vecu}{\mathbf{u}}
\newcommand{\dm}{\mathrm dm_{n,p}^\hbar}
\newcommand{\dz}{\mathrm d\vecz}
\newcommand{\dzb}{\mathrm d\overline\vecz}
\newcommand{\E}{\mathcal E_{n,p}}
\newcommand{\T}{{\bf T}_{n,p}}
\newcommand{\To}{{\mathcal T}}
\newcommand{\F}{{}_0\mathrm F_1}

\newcommand{\Sn}{\mathbf{S}^n}
\newcommand{\vecx}{\mathbf{x}}
\newcommand{\vecy}{\mathbf{y}}
\newcommand{\vtheta}{\boldsymbol{\theta}}
\newcommand{\ds}{\mathrm d\mathbf S_n}
\newcommand{\bk}{{\bf k}}
\newcommand{\bm}{{\bf m}}
\newcommand{\K}{{{\bf K}_{n,p}^\hbar}}
\newcommand{\Ku}{{{\bf K}_{n,-1}^\hbar}}

\newcommand{\nrs}{{{\bf H}_{n}}}
\newcommand{\B}{{{\mathfrak B}_{\hbar,p}}}
\newcommand{\Be}{{{\mathfrak B}}}
\newcommand{\Pro}{{{\mathbf P}}}

\newcommand{\balpha}{\boldsymbol \alpha}
\newcommand{\bbeta}{\boldsymbol \beta}

\newcommand{\U}{{\mathbf U_{n,p}}}
\newcommand{\A}{{\mathcal A_{n,p}^\hbar}}

\begin{document}
\title[Berezin symbols on the complex unit sphere]
{Berezin symbols of operators on the  unit sphere of $\mathbb C^n$}
\author{Erik I. D\'{\i}az-Ort\'iz}
\address{CONACYT Research Fellow – Universidad Pedagógica Nacional - Unidad 201 Oaxaca}
\email{eidiazor@conacyt.mx}

\maketitle
\tableofcontents
\begin{abstract}
We describe the symbolic calculus of  operators on the unit sphere in the complex $n$-space $\mathbb C^n$ defined by the Berezin quantization. In particular, we derive a explicit formula for the composition of Berezin symbol and with that a noncommutative star product. In the way is necessary introduce a  holomorphic spaces  which admit  a reproducing kernel in the form of generalized hypergeometric series.
\end{abstract}

\section{Introduction and summary}
We start recalling some results of Berezin's theory that will be used below. See Ref. \cite{B74} for details.

Let $H$ be a Hilbert space endowed with the inner product  $(\cdot,\cdot)$ and $M$ some set with the measure $\mathrm d \mu$. Let $\{e_\alpha \in H\;|\; \alpha \in M\}$ a family of functions in $H$ labelled by elements of $M$, such that satisfies the following properties:
\begin{enumerate}
\item[a) ] The family $\{e_\alpha\}$ is complete, this is for any $f,g \in H$ Parseval's identity is valid
\begin{displaymath}
(f,g)= \int_M (f,e_\alpha)(e_\alpha,g) \mathrm d \mu(\alpha)\;.
\end{displaymath}
\item[b) ] The map $f \to \hat f$, defining by $\hat f(\alpha)= (f,e_\alpha)$, is an embedding from $H$ into $L^2(M,\mathrm d\mu)$.
\end{enumerate}

In 1970’s, Berezin \cite{B74} introduced a general symbolic calculus for bounded linear operators on $H$. More specifically, for $A \in \mathbf B (H)$, the algebra of all bounded linear operator on $H$, the Berezin symbol (or Berezin transform) of $A$ is the function on $M$ defined by
\begin{displaymath}
\Be( A) (\alpha)= \frac{(Ae_\alpha,e_\alpha)}{(e_\alpha,e_\alpha)}\;, \hspace{0.5cm} \alpha \in M\;.
\end{displaymath}
The prototypes of the spaces $H$ are the Bergman spaces of all holomorphic functions in $L^2(M, \mathrm d\mu )$ on a bounded domain $M\subset \mathbb C^n$ with Lebesgue measure $\mathrm d\mu$, or the Segal-Bargmann spaces of all entire functions in $L^2(\mathbb C^n,\mathrm d \mu)$ for the Gaussian measure $\mathrm d\mu (z) = (2\pi)^{-n} e^{-|z|^2/2}\mathrm dz \mathrm d\overline z$, where $\mathrm dz \mathrm d\overline z$ denoting Lebesgue measure on $\mathbb C^n$. In these cases the functions $e_\alpha$ are the reproducing kernel.

Moreover, for every $A,B \in \mathbf B(H)$ and $f \in H$, we have the following formulas
\begin{align}
\Be(A+cB) & = \Be( A) + c \Be( B)\;, \mbox{ for all constant } c \in \mathbb C\;,\nonumber\\
\Be(\mathrm{Id}) & =1\;,\hspace{0.5cm} \mbox{with $\mathrm{Id}$ the identity operator }.\nonumber
\end{align}
If we suppose that the Berezin's symbol may be extended in a neighbourhood of the diagonal $M \times M$ to the function
\begin{displaymath}
\Be( A)(\alpha,\beta)= \frac{(Ae_\alpha,e_\beta)}{(e_\alpha,e_\beta)}\;,
\end{displaymath}
then we have the following formulas
\begin{align}
\Be( A^*)(\alpha,\beta) & = \overline{\Be(A)(\beta,\alpha)} \;, \hspace{0.5cm} \alpha,\beta\in M\;.\nonumber\\
\widehat{Af}(\alpha) & = \int_M \hat f(\beta) \Be(A)(\beta,\alpha) (e_\beta,e_\alpha) \mathrm d \mu(\beta)\;, \label{Eq 1}\\
\Be(AB)(\alpha,\beta) & = \int_M \Be(A)(\gamma,\beta) \Be(B)(\alpha,\gamma) \frac{(e_\alpha,e_\gamma)(e_\gamma,e_\beta)}{(e_\alpha,e_\beta)} \mathrm d\mu(\gamma)\;. \label{Eq 2}
\end{align}

The useful application of this symbolic calculus is that allows us to build a star product. 
In \cite{B74} Berezin applied this method to K\"ahler manifolds, in this case $H$ is the Hilbert space of functions in $L^2(M,\mathrm d\mu)$ which are analytic so that the embedding from $H$ into $L^2(M,\mathrm d\mu)$ is the inclusion, and the complete family $\{e_\alpha\}$ are the reproducing kernel with one variable fixed.

The main goal of the present paper is to introduce a Berezin symbolic calculus for the Hilbert space $\mathcal O$ of all functions in $L^2(\Sn)$ whose Poisson extension into the interior of $\Sn$ is holomorphic, where $\Sn=\{\vecx\in\mathbb C^n\;|\; |x_1|^2+\cdots+|x_n|^2=1\}$ and  $L^2(\Sn)$ denote the Hilbert space of the square integrable function with respect to the normalized surface measure $\ds(\vecx)$ on $\Sn$ and endowed with the usual inner product
\begin{equation}
 \langle \phi,\psi\rangle _{\Sn}= \int_{\Sn} \phi(\vecx) \overline{\psi(\vecx)} \ds(\vecx)\;,\;\hspace{0.5cm} \phi,\psi \in L^2(\Sn)\;.
\end{equation}

To achieve this,  we build a family $\{e_\alpha\} $ of functions in $\mathcal O $ satisfying a), since coherent states meet this property and based of our experience in $ L^2(S^m) $, $ m = 2,3,5 $, where $ S^m =\{\vecy \in \mathbb R^{m + 1} \; | \; y_1^2 + \cdots + y_{m + 1}^2 = 1 \} $ (see Ref. \cite{D-V09}), we propose in Sec. \ref{class of holomorphic spaces} our coherent states in $ \mathcal O $ as a suitable power series of the inner product $\sum_{\ell=1}^n x_\ell \overline z_\ell/ \hbar $ regarding $ (x_1,\ldots,x_n) \in \Sn $ and $ (z_1,\ldots,z_n) \in \mathbb C^n $, where $ \hbar $ denoting the Planck's constant. In addition, to ensure b) we define a new Hilbert space $ L^2(\mathbb C^n, \dm) $, whose measure $ \dm $ is obtained such that the inner product of these coherent states is the reproducing kernel from range of mapping $ f \to \hat f $ (see Sec. \ref{section construction measure} for details of how this measure was obtained and  Eq. (\ref{measure on U}) for definition of $\dm$).

In order to prove that the coherent states and the space $L^2(\mathbb C^n,\dm)$ constructed in Sec. \ref{class of holomorphic spaces} satisfy the conditions a) and b), in Sec. \ref{space En} we prove that the mapping $ f \to \hat f $, denoted by $\U$, is an isometry from $\mathcal O$ to $\E \subset L^2(\mathbb C^n, \dm)$. In addition, in Sec. \ref{section coherent states} we prove that the unitary transformation $\U$ applied to our coherent states gives the reproducing kernel of the space $\E$, which allows us to prove that the family of coherent states form a complete system for $\mathcal O$.

From this and according to Berezin's theory, in Sec. \ref{section Berezin symbolic calculus} we describe the rules for symbolic calculus on $\mathbf B(\Sn)$, further more we obtain asymptotic expansions of Berezin's symbol of Toeplitz operator.

Finally,  in Sec. \ref{section star product} we define a star product on the algebra out of Berezin's symbol for bounded operators with domain in $\mathcal O$ and will prove that this noncommutative star-product satisfies the usual requirement on the semiclassical limit.

It should be noted, since $ H = \mathcal O $ and $ L^2 (M, \mathrm d \mu) = L^2 (\mathbb C^n, \dm) $, in our construction there is no inclusion of $ H $ into $ L^2 (M, \mathrm d \mu) $, further more, the complete family $\{ e_\alpha\} $ is not obtained by the reproducing kernel. This situation is thus slightly different from Berezin's one.


In the present paper we  will be used throughout the text the following  basic notation.
For every $\vecz, \vecw \in \mathbb C^n$, $\vecz=(z_1,\ldots,z_n)$, $\vecw=(w_1,\ldots,w_n)$, let
\begin{displaymath}
\vecz\cdot\vecw=\sum_{\ell=1}^n z_\ell\,\overline w_\ell\;, \hspace{0.8cm} |\vecz|=\sqrt{\vecz \cdot \vecz}\;.
\end{displaymath}
Furthermore, for every multi-index $\bk=(k_1,\ldots,k_n)\in \mathbb Z_+^n$ of length $n$ where $\mathbb Z_+$ is the set of non negative integers, let \begin{displaymath}
|\bk|=\sum_{\ell=1}^n k_\ell\;, \hspace{0.8cm} \bk!=\prod_{\ell=1}^n k_\ell! \;, \hspace{0.8cm} \vecz^\bk=\prod_{\ell=1}^n z_\ell^{k_\ell}\;.
\end{displaymath}
Given $\omega\in \mathbb C$, let us denote its real and imaginary parts by $\Re(\omega)$ and $\Im(\omega)$ respectively.



\section{A class of holomorphic spaces} \label{class of holomorphic spaces}
In this section, we will give a description of how to define the special Hilbert space $L^2(\mathbb C^n,\dm)$. 

\subsection{Construction of measure $\dm$.}\label{section construction measure}

Based on the expression of coherent states obtained for spaces $L^2(S^m)$, with $m=2,3,5$ and $S^m=\{\vecx \in \mathbb R^{m+1}\;|\; x_1^2+\ldots+x_{m+1}^2=1\}$ (See \cite{D-V09}  for more details). We consider the coherent states on $L^2(\Sn)$ as the functions
\begin{equation}\label{first definition coherent states}
\Phi_{\vecw}(\vecx)=\sum_{\ell=0}^\infty \frac{c_\ell}{\ell!} \left(\frac{\vecx\cdot \vecw}{\hbar}\right)^\ell\;,\hspace{0.5cm} \vecx\in \Sn\;, \vecw\in \mathbb C^n\;,
\end{equation}
where the constants $c_\ell$ are defined below, which will be obtained considering that $\langle \Phi_\vecz,\Phi_\vecw\rangle_{\Sn}$ is the reproducing kernel of some Hilbert space.

From Lemma \ref{product in Sn}, we have
\begin{align*}
\langle\Phi_\vecz,\Phi_\vecw\rangle_{\Sn} & =\Gamma(n)\sum_{\ell=0}^\infty c_\ell^2 \frac{\displaystyle\left(\frac{\vecz\cdot\vecw}{\hbar^2}\right)^\ell}{\ell!\;\Gamma(\ell+n)}\;,
\end{align*}
where $\Gamma$ denote the Gamma function. If the constants $c_\ell$ take the following values
\begin{displaymath}
c_\ell^2=\frac{C\;\Gamma(n+\ell)}{\Gamma(n+p+\ell)\Gamma(n)}\;,\hspace{0.5cm} \mbox{with }p \in \mathbb R,
\end{displaymath}
and $C$ a constant, then we have
\begin{equation}\label{help 1}
\langle\Phi_\vecz,\Phi_\vecw\rangle_{\Sn}= C\left(\frac{\vecz\cdot\vecw}{\hbar^2}\right)^{\frac{1}{2}(1-p-n)} \mathrm I_{n+p-1}\left(\frac{2\sqrt{\vecz\cdot\vecw}}{\hbar}\right)
\end{equation}
with $\mathrm I_k$ denoting the modified Bessel function of the first kind of order $k$ (see Secs. 8.4 and 8.5 of Ref. \cite{G94} for definition and expressions for this special function). We are taking the branch of the square root function such that $\sqrt z=|z|^{1/2}\exp(\imath \theta/2)$, where $\theta= \mathrm{Arg}(z)$ and $-\pi < \theta <  \pi$. 

The Aronszajn-Moore theorem \cite{A50} states that,  every positive definite (Hermitian) function $K$ on $\mathbb C^n \times \mathbb C^n$ determines a unique Hilbert function space $H$ for which $K$ is the reproducing kernel. 

In order to construct the Hilbert space $H$ for which the right side of  Eq. (\ref{help 1}) is its reproducing kernel, let $\phi(z)=C(z/\hbar)^{-\nu/2}\mathrm I_\nu(2\sqrt z/\hbar)=\sum_{\ell=0}^\infty b_\ell z^\ell$. Define the space $H_\phi$  as the set of  all holomorphic functions in $\mathbb C^n$ equipped with inner product
\begin{equation}\label{help inner product}
(f,g)_{H_\phi} = \sum_{\ell=0}^\infty \frac{1}{b_\ell}\sum_{|\bk|=\ell} \frac{\bk!}{\ell!}f_\bk \overline{g_\bk}\;,
\end{equation}
where $f_\bk$, $g_\bk$ are Taylor's coefficients of $f$ and $g$, respectively.

Note that the function $\phi(\vecz\cdot\vecw)$ belongs to $H_\phi$ as a function of $\vecz\in \mathbb C^n$ for every fixed $\vecw\in \mathbb C^n$, and for any $f\in H_\phi$
\begin{equation*}
(f,\phi((\cdot)\cdot \vecw))_{H_\phi}= \sum_{\ell=0}^\infty \frac{1}{b_\ell}\sum_{|\bk|=\ell} \frac{\bk!}{\ell!}f_\bk \overline{\frac{b_\ell \ell!}{\bk!} \overline{\vecw^\bk}}= \sum_{\ell=0}^\infty \frac{1}{b_\ell}\sum_{|\bk|=\ell} \frac{\bk!}{\ell!}f_\bk \vecw^\bk=f(\vecw)\;.
\end{equation*}
Thus, $\phi(\vecz\cdot \vecw)$ is the reproducing kernel of $H_\phi$. From  (\ref{help inner product}) we see that
\begin{equation}\label{measure 1}
(\vecz^\bk,\vecz^\bk)_{H_\phi} = \frac{1}{b_{|\bk|}} \frac{\bk!}{|\bk|!}\;.
\end{equation}
We now assume that the inner product (\ref{help inner product}) can be expressed by
\begin{equation}\label{measure 2}
(f,g)_{H_\phi}=\int_{\mathbb C^n} f(\vecz)\overline{g(\vecz)} \omega(|\vecz|) \dz \dzb\;.
\end{equation}
Substituting the value of $b_{|\bk|}$ in Eq. (\ref{measure 1}), using Eq.  (\ref{measure 2}), polar coordinates and Lemma \ref{product in Sn} we obtain
\begin{align*}
\frac{1}{C}\hbar^{2|\bk|} \bk! \Gamma(\nu+|\bk|+1) & = (\vecz^\bk,\vecz^\bk)_{H_\phi} \\
& = \frac{2\pi^n \bk!}{(n-1+|\bk|!)} \int_0^\infty r^{2n-1+2|\bk|} \omega(r) \mathrm dr\;.
\end{align*}
So, we need find $\omega(r)$ that satisfies the equation
\begin{equation} \label{measure 3}
\int_0^\infty r^{2n-1+2|\bk|} \omega(r) \mathrm dr=\frac{\hbar^{2|\bk|}}{2C\pi^n} \Gamma(\nu+|\bk|+1) \Gamma(n+|\bk|)\;.
\end{equation}
Note that expression (\ref{measure 3}) becomes essentially the Mellin transform (see Refs. \cite{G94}, \cite{P-B-M90}). From this simple observation and the formula 6.561-16  of Ref \cite{G94} we immediately obtain 
\begin{equation}\label{measure}
\omega(r)=\frac{2}{C(\pi\hbar^2)^n}\left( \frac{r}{\hbar}\right)^{\nu-n+1} \mathrm K_{\nu-n+1}\left(\frac{2}{\hbar}r\right)\;,\hspace{0.5cm} \nu>-1\;,
\end{equation}
with $\mathrm K_p$ denoting the MacDonal-Bessel functions of order $p$ (see Secs. 8.4 and 8.5 of Ref. \cite{G94} for definitions and expressions for this special functions). Taking the constant $C$ such that $(1,1)_{H_\phi}=1$ we obtain 
\begin{equation}\label{constant C}
C=\Gamma(\nu +1)\;.
\end{equation}

\subsection{The space $\E$}
In the previous section we constructed a Hilbert space whose inner product can be expressed in the form (\ref{measure 2})  and its reproducing kernel is the right side of Eq. (\ref{help 1}); in this section we proves rigorously all those results. From Eqs. (\ref{measure}) and (\ref{constant C}) let us consider the following measure $\dm$ on $\mathbb C^n$
\begin{equation}\label{measure on U}
 \dm(\vecz)=\frac{1}{\Gamma(n+p)}\frac{2}{(\pi\hbar^2)^n}\left(\frac{|\vecz|}{\hbar}\right)^p \mathrm K_p\left(2\frac{|\vecz|}{\hbar}\right) \dz \dzb \;, \hspace{0.5cm} p>-n
\end{equation}
with $\vecz=(z_1,\ldots,z_n)\in \mathbb C^n$ and $\dz\dzb$ denoting Lebesgue measure on $\mathrm C^n$.

Note that the measure $\dm$ is invariant under the rotation group
$\mathrm{SO}(n)$ (the group of $n \times n$ orthogonal matrices with unit determinant and real entries). The action of $\mathrm{SO}(n)$ on $\mathbb C^n$
that we are considering is the natural extension of the usual action of $\mathrm{SO}(n)$ on $\mathbb R^n$, this is: for $R \in \mathrm {SO}(n)$ defining $\mathrm T_R: \mathbb C^n \to \mathbb C^n$ by $\mathrm T_R(\vecz)=R \;\Re (\vecz)+ \imath R\; \Im(\vecz)$.

Let us consider the Hilbert space $L^2(\mathbb C^n,\dm)$ of square integrable functions on $\mathbb C^n$ with respect to the measure $\dm$ and endowed with the inner product
\begin{equation}
 (f,g)_p = \int\limits_{\vecz \in \mathbb C^n} f(\vecz)\overline{g(\vecz)} \dm(\vecz)\;,\hspace{0.5cm} f,g \in L^2(\mathbb C^n , \dm)\;.
\end{equation}

Let us denote by $||f||_p= \sqrt{(f,f)_p}$ the corresponding norm of $f \in L^2(\mathbb C^n , \dm)$.

\begin{definition}
 For $p>-n$, the space of entire functions $f$ defined on $\mathbb C^n$ such that $||f||_p$ is finite is denoted by $\E$.
\end{definition}

\begin{remark}
For $n=1$, the spaces $\mathcal E_{1,p}$ were used by Karp Dmitrii to derive an analytic continuation formula for functions on $\mathbb R^+$ (see \cite{K03}).
\end{remark}

\begin{lemma}\label{product in Enp}
 Let 
 \begin{equation*}
  \mathrm L:=\int\limits_{\mathbb C^n} \vecz^{\bk} \overline\vecz^\bm (\vecz \cdot \vecw)^s (\vecu \cdot \vecz)^\ell \dm(\vecz)\;,
 \end{equation*}
 with $\ell,s \in \mathbb Z_+$, $\bk,\bm \in \mathbb Z_+^n$, and $\vecw,\vecu \in \mathbb C^n\backslash\{0\}$. Then $\mathrm L=0$ if $|\bk|+s\ne|\bm|+\ell$, and if $|\bk|+s=|\bm|+\ell$, then 
 \begin{equation}
  \mathrm L= \hbar^{2(|\bk|+s)}(n)_{|\bk|+s}(n+p)_{|\bk|+s} \int_{\Sn} \vecx^\bk \overline \vecx^\bm (\vecx \cdot \vecw)^s (\vecu \cdot \vecx)^\ell \ds(\vecx)\;.
 \end{equation} 
\end{lemma}

\begin{proof} From Eq. (\ref{measure on U}) and expressing L in polar coordinates 
\begin{equation*}
 \mathrm L= \frac{|\Sn|2\hbar^{-(p+2n)}}{\pi^n\Gamma(n+p)}\int\limits_0^\infty r^{2n-1+|\bm|+|\bk|+\ell+s+p} \mathrm K_p(2\frac{r}{\hbar}) \mathrm dr \int\limits_{\Sn} \vecx^\bk \overline \vecx^\bm (\vecx \cdot \vecw)^s (\vecu \cdot \vecx)^\ell \ds(\vecx)\;,
\end{equation*}
where $|\Sn|=\mathrm{vol}(\Sn)$. From the formula 6.561-16 of Ref \cite{G94} and Lemma \ref{product in Sn} we conclude the proof of Lemma \ref{product in Enp} 
\end{proof}

\begin{prop} Let  $p>-n$ 
 \begin{enumerate}
  \item[a) ] The space of analytic polynomials defined on $\mathbb C^n$ is dense on $\E$.
  
  \item[b) ] For $\ell \in \mathbb Z_+$, let us denote by $W_\ell$ the space of homogeneous polynomials of degree $\ell$. Then 
  \begin{equation}\label{direct sum E}
   \E = \bigoplus_{\ell=0}^\infty W_\ell\;. 
  \end{equation}
  Further, the set $\{\Phi_{\bk,p}^\hbar\;|\; \bk \in \mathbb Z_+^n\}$ is orthonormal basis of $\E$, where $\Phi_{\bk,p}^\hbar: \mathbb C^n \to \mathbb C$ is defining by
  \begin{equation}\label{Phi}
   \Phi_{\bk,p}^\hbar(\vecz)=\frac{1}{\hbar^{|\bk|}} \sqrt{\frac{1}{(n+p)_{|\bk|}\bk!}} \;\vecz^{\bk}\;,
  \end{equation}
where, for $b> 0$, the Pochhammer symbol is given by
 \begin{equation*}
  (b)_\ell=\frac{\Gamma(b+\ell)}{\Gamma(b)}\;.
 \end{equation*}

  \item[c) ] The space $\E$ enjoys the property of having a reproducing kernel $\T=\T(\vecz,\vecw)$. Namely, for all $f \in \E$ we have
  \begin{equation}\label{property kernel E}
   f(\vecz)=(f,\T(\cdot,\vecz))=\int\limits_{\vecw \in \mathbb C^n} f(\vecw) \T(\vecz,\vecw) \dm(\vecw)\;,\hspace{0.5cm} \forall\vecz \in \mathbb C^n
  \end{equation}
  where the reproducing kernel is given by
  \begin{equation}\label{kernel}
   \T(\vecz,\vecw)= \Gamma(n+p) \left(\frac{\vecz \cdot \vecw}{\hbar^2}\right)^{\frac{1}{2}(-p-n+1)} \mathrm I_{n+p-1} \left(\frac{2 \sqrt{\vecz \cdot \vecw}}{\hbar}\right).
  \end{equation}
Even more 
  \begin{equation}\label{kernel series}
   \T(\vecz,\vecw)= \sum_{\ell=0}^\infty \mathbf T_{n,p}^\ell(\vecz,\vecw)\;,
  \end{equation}
  where $\mathbf T_{n,p}^\ell$ are the reproducing kernel of the subspace $W_\ell$
  \begin{displaymath}
   \mathbf T_{n,p}^\ell(\vecz,\vecw)=\frac{1}{\ell! (n+p)_\ell}\left(\frac{\vecz \cdot\vecw}{\hbar^2}\right)^\ell\;.
  \end{displaymath}

    
 \end{enumerate}
\end{prop}
\begin{proof} a) Is a consequence of the Stone-Weiertrass theorem.

b) Note that the space $W_\ell$ is contained in $L^2(\mathbb C^n, \dm)$. Then by a), we only need prove that the spaces $W_\ell$ and  $W_{\ell'}$ are orthogonal for $\ell \ne \ell'$, 
which is a consequence of Lemma \ref{product in Enp}.
The second part of this point too is a direct consequence of Lemma \ref{product in Enp}, since
\begin{equation}
 \left\{\frac{1}{\hbar^{\ell}} \sqrt{\frac{1}{(n+p)_{\ell}\bk!}} \;\vecz^{\bk}\;|\; \bk \in \mathbb Z_+^n\;,\;|\bk|=\ell\right\}
\end{equation}
is orthogonal basis of $W_\ell$.

c) By Eq. (\ref{direct sum E}), it suffices to show that the function $\mathbf T_{n,p}^\ell(\vecz,\vecw)$ is the reproducing kernel in $W_\ell$. That will appear as a simple consequence of Lemma \ref{product in Enp}.  
\end{proof}

As a consequence of the existence of the reproducing kernel and the Cauchy-Schwartz inequality, we obtain the following estimate for any $f \in \E$
\begin{displaymath}
 |f(\vecz)|\le ||f||_p\;||\T(\cdot,\vecz)||\;,\hspace{0.5cm} \forall \vecz \in \mathbb C^n\;.
\end{displaymath}

\section{Unitary transform} \label{space En}

In this section we introduce a unitary transform $\U$ from $\mathcal O$ onto Hilbert space $\E$, with $n \ge 2$ and $p > -n$. In order to define this transformation, we first recall some results about the space $\mathcal O$. 

For $\ell \in \mathbb Z_+$, let us define the space $V_\ell$ of homogeneous polynomials of degree $\ell$ as the vector space of restrictions to the $n$-sphere of homogeneous polynomials of degree $\ell$ defined initially on the ambient space $\mathbb C^n$. We will use the fact that $\mathcal O$ is equal to the direct sum of the space $V_\ell$ (see \cite{R08})
\begin{equation}
 \mathcal O= \bigoplus_{\ell=0}^\infty V_\ell\;.
\end{equation} 
Further, the set $\left\{\phi_{\bk}\;|\; \bk\in \mathbb Z_+^n\right\}$ is orthogonal basis of $\mathcal O$, where $\phi_{\bk}:\Sn \to \mathbb C$ is defining by
\begin{equation}\label{phi}
 \phi_{\bk}(\vecx)= \sqrt{\frac{(n)_{|\bk|}}{\bk!}}
 \;\vecx^\bk\;.
\end{equation}

For every $\ell \in \mathbb Z_+$, exist a natural transformation from $V_\ell$ to $W_\ell$, this is the linear extension of the assignment $\phi_{\bk} \to \Phi_{\bk,p}^\hbar$, where $\phi_{\bk}$, $\Phi_{\bk,p}^\hbar$ are defined in Eqs. (\ref{phi}) and (\ref{Phi}) respectively, and $\bk \in \mathbb Z_+^n$ with $|\bk|=\ell$. Let us define this unitary transform by $\mathbf U_{n,p}^\ell$.

We now consider the linear extension $\U:\mathcal O \to \E$ of the operators $\mathbf U_{n,p}^\ell$ defined as follows:

Given $\Psi \in \mathcal O$ written as $\Psi=\sum_{\ell=0}^\infty \Psi_\ell$ with $\Psi_\ell \in V_\ell$, we define the partial sums $S_k(\Psi)=\sum_{\ell=0}^k \mathbf U_{n,p}^\ell\Psi_\ell$ and then
\begin{equation}\label{definition U limits}
 \U \Psi := \lim_{k \to \infty} S_k(\Psi)\;.
\end{equation}

The operator $\U$ is well defined and unitary due to the unitarity of the operators $\mathbf U_{n,p}^\ell$ and the fact that every element $f$ in the space $\E$ can be written in a unique way as $f=\lim_{k\to \infty} \sum_{\ell=0}^k f_\ell$ with $f_\ell \in W_\ell$. (See Eq. (\ref{direct sum E})).


\begin{theorem}\label{unitary}
 The operator $\U$ mapping $\mathcal O$ isometrically onto the space $\E$. Even more
 \begin{equation}\label{expression U}
  \U\Psi(\vecz)= 
  \int_{\Sn} 
  \Psi(\vecx) \sum_{\ell=0}^\infty \frac{c_{\ell,p}}{\ell!}\left(\frac{\vecz \cdot \vecx}{\hbar}\right)^\ell\
  \ds(\vecx)\;,\hspace{0.5cm} \forall \Psi \in \mathcal O\;,\vecz \in\mathbb C^n\;,
 \end{equation}
 where
 \begin{equation}\label{constants coherent states}
  (c_{\ell,p})^2= \frac{(n)_\ell}{(n+p)_\ell}\;.
 \end{equation}
 
\end{theorem}

\begin{proof}
Note that for $\vecz \in \mathbb C^n$ fixed and $\vecx \in \Sn$, the series 
\begin{equation}\label{coherent states}
 \K(\vecx,\vecz):=\sum_{\ell=0}^\infty \frac{c_{\ell,p}}{\ell!}\left(\frac{\vecx \cdot \vecz}{\hbar}\right)^\ell\
\end{equation}
is bounded on $\Sn$ and therefore it is a vector in $L^2(\Sn)$ as well. Then, by the Cauchy-Schwartz inequality, the right side of Eq. (\ref{expression U}) is a well defined function. 

Let $\ell \in \mathbb Z_+$, is not difficult show that
 \begin{equation}\label{expresion Ul integral}
  \mathbf U_{n,p}^\ell \Phi(\vecz)=\int_{\Sn} \sum_{|\bk|=\ell} \overline{\phi_{\bk}(\vecx)}\Phi_{\bk,p}^{\hbar}(\vecz) \Phi(\vecx) \ds(\vecx)\;,\hspace{0.5cm} \vecz \in \mathbb C^n\;,\hspace{0.5cm} \Phi\in V_\ell\;.
 \end{equation}
 
Even more, by Eqs (\ref{phi}) and (\ref{Phi})
\begin{align}
 \sum_{|\bk|=\ell} \overline{\phi_{\bk}(\vecx)} \Phi_{\bk, p}^{\hbar}(\vecz) & = \sum_{|\bk|=\ell} 
 \frac{1}{\hbar^\ell}\sqrt{\frac{(n)_\ell}{\bk!}} \sqrt{\frac{1}{(n+p)_\ell\bk!}}\; \overline \vecx^\bk \vecz^\bk\nonumber\\
 & = \frac{c_{\ell,p}}{\ell!}\left(\frac{\vecz \cdot \vecx}{\hbar}\right)^\ell\;. \label{kernel Ul}
\end{align}
Then, given $\Psi \in \mathcal O$, written as $\Psi=\sum_{\ell=0}^\infty \Psi_\ell$ with $\Psi_\ell \in V_\ell$, we obtain the Eq. (\ref{expression U}) from definition of $\U$ (see Eq. (\ref{definition U limits})), the integral expression the operator $\mathbf U_{n,p}^\ell$ (see Eqs. (\ref{expresion Ul integral}) and (\ref{kernel Ul})) and the dominated converge theorem.\end{proof}
 
We end this section with an inversion formula for the operator $\U$.
\begin{theorem}\label{inverse}
 Let $f \in \E$, then
 \begin{equation*}
  (\U)^{-1} f(\vecx)= \int_{\mathbb C^n}  \sum_{\ell=0}^\infty \frac{c_{\ell,p}}{\ell!}\left(\frac{\vecx \cdot \vecz}{\hbar}\right)^\ell f(\vecz) \dm(\vecz)\;,\hspace{0.5cm}\vecx \in \Sn\;.
 \end{equation*}
\end{theorem}

\begin{proof} Since the set $\{\phi_\bk\;|\;\bk\in \mathbb Z_+^n\}$ is orthonormal basis of $\mathcal O$, then the reproducing kernel,  $\nrs$, from $\mathcal O$ is
\begin{displaymath}
\nrs(\vecx,\vecy)=\sum_{\ell=0}^\infty\sum_{|\bk|=\ell} \phi_\bk(\vecx) \overline{\phi_\bk(\vecy)}=\sum_{\ell=0}^\infty \frac{(n)_\ell}{\ell!}(\vecx \cdot\vecy)^\ell = {}_2F_1(n,1;1;\vecx\cdot\vecy),
\end{displaymath}
 where ${}_2F_1$ denoting the Gauss hypergeometric function (see Secs. 9.1 of Ref.  \cite{G94} for definition and expression for this special function).
 
Even more, for $\vecy \in \Sn$ fixed, we obtain from Lemma \ref{product in Sn}
\begin{equation}\label{U reproducing kernel Sn}
\U \nrs(\cdot,\vecy)(\vecz) = \K(\vecz,\vecy)\;\hspace{0.5cm} \forall \vecz\in \mathbb C^n\;.
\end{equation} 

Moreover, from the reproducing property of $\nrs$, the unitarity of the transform $\U$ and the Eq. (\ref{U reproducing kernel Sn}), we have
\begin{align*}
(\U)^{-1}f(\vecx) & = \langle (\U)^{-1}f, \nrs(\cdot,\vecx)\rangle_{\Sn}\\
& = (f,\U\nrs(\cdot,\vecx))_p\\
& = (f,\K(\cdot,\vecx))_p 
\end{align*}
\end{proof}

\section{Coherent states for $\mathcal O$}\label{section coherent states}
 
 We consider the set of functions $\mathcal K=\{\K(\cdot,\vecz)\;|\; \vecz \in \mathbb C^n\} \subset \mathcal O$ defined by Eq. (\ref{coherent states}).
In  this section shows that the functions in $\mathcal K$  satisfy the conditions to define a Berezin symbolic calculus on $\mathcal O$.

First notice that the unitary operator $\U$ is a coherent states transform because its action on a function $\Phi$, in $\mathcal O$, is a function in $\E$ whose evaluation in $\vecz\in \mathbb C^n$ is equal to the $L^2(\Sn)$-inner product of $\Phi$ with the coherent states labeled by $\vecz$. This is
\begin{equation}\label{U inner product coherent states}
\U\Phi(\vecz)= \langle \Phi,\K(\cdot,\vecz)\rangle _{{\Sn}}\;.
\end{equation}
\begin{remark} 
 From Eq. (\ref{U inner product coherent states}) and theorem \ref{inverse} is not difficult to prove
\begin{equation*}\label{resolution of the identity for Sn}
\Psi(\vecx)= \int_{\mathbb C^n} \langle \Psi,\K(\vecz)\rangle _{_{\Sn}} \K(\vecx,\vecz)\dm(\vecz)\;,\hspace{0.5cm} \vecx\in\Sn\;,\Psi\in \mathcal O\;.
\end{equation*}
Thus the system of coherent states provides a resolution of identity for $\mathcal O$, this is the so-called reproducing property of the coherent states. 
\end{remark}

 
 
 \begin{theorem}\label{U coherent states}
  The unitary operator $\U$ applied to $\K(\cdot,\vecz)$ gives the reproducing kernel of the Hilbert space $\E$, $n \ge 2$, $p>-n$. Namely, for all $\vecz ,\vecw$ in $\mathbb C^n$ we have
  \begin{displaymath}
   [\U\K(\cdot,\vecz)](\vecw)=\T(\vecw,\vecz)\;.
  \end{displaymath}
 \end{theorem}
\begin{proof} From theorem \ref{unitary}, the dominated converge theorem and Lemma \ref{product in Sn} we obtain
\begin{align}
 [\U\K(\cdot,\vecz)](\vecw) & = \int_{\Sn} \K(\vecx,\vecz) \overline{\K(\vecx,\vecw)} \ds(\vecx)\nonumber\\
 & = \sum_{\ell=0}^\infty \left(\frac{c_{\ell,p}}{\ell!\hbar^{\ell}}\right)^2 \int_{\Sn}\left(\vecx\cdot\vecz\right)^\ell \left(\vecw\cdot\vecx\right)^\ell \ds(\vecx)\nonumber\\
 & = \sum_{\ell=0}^\infty \frac{1}{(n+p)_\ell}\frac{1}{\ell!} \left(\frac{\vecw\cdot \vecz}{\hbar^2}\right)^\ell .\nonumber
\end{align}
And this is the expression for the reproducing kernel $\T$ given in Eq. (\ref{kernel series}) as an infinity series.
\end{proof}


The proof of some theorems below uses an estimate of the norm of a coherent state $\K(\cdot,\vecz)$  which in turn is a consequence of the following estimate for the inner product of two coherent states:

\begin{prop}\label{proposition inner product coherent states}
 Let $\vecz,\vecw \in \mathbb C^n$ fixed. Assume $\vecz \cdot \vecw \ne 0$ and $|\mathrm{Arg}(\vecz \cdot \vecw)| <  \pi$. Then for $\hbar \to 0$
 \begin{align}
  \langle \K(\cdot,\vecz),\K(\cdot,\vecw)\rangle _{{\Sn}}= & \frac{\Gamma(n+p)}{2\sqrt \pi} \left(\frac{\vecw \cdot \vecz}{\hbar^2}\right)^{\frac{1}{2}(-p-n+\frac{1}{2})} \mathrm{e}^{\frac{2}{\hbar}\sqrt{\vecw \cdot \vecz}}\biggl[1 -\nonumber\\
& \left.\frac{(n+p-\frac{3}{2})(n+p-\frac{1}{2})}{4\sqrt{\vecw\cdot\vecz}}\hbar+\mathrm O(\hbar^2)\right]\;.\label{asymptotic inner product coherent states}
 \end{align}
\end{prop}

\begin{proof} Since the transform $\U$ is a unitary transformation, we obtain from theorem \ref{U coherent states}, the reproducing property of $\T$ (see Eq. (\ref{property kernel E})) and the expression for the reproducing kernel $\T$ given in Eq. (\ref{kernel}) that
\begin{align}
\langle \K(\cdot,\vecz),\K(\cdot,\vecw)\rangle _{{\Sn}} & =\T(\vecw,\vecz)\nonumber\\
& = \Gamma(n+p) \left(\frac{\vecw \cdot \vecz}{\hbar^2}\right)^{\frac{1}{2}(-p-n+1)}\mathrm I_{n+p-1}\left(\frac{2\sqrt{\vecw\cdot\vecz}}{\hbar}\right).\label{inner product coherent states}
\end{align}
The modified Bessel function $\mathrm I_\vartheta$, $\vartheta\in\mathbb R$, has the following asymptotic expression when $|\omega| \to \infty$ (see formula 8.451-5 of Ref. \cite{G94})
\begin{equation}\label{asymptotic expression modified Bessel function}
\mathrm I_\vartheta(\omega)=\frac{\mathrm e^\omega}{\sqrt{2\pi\omega}}\sum_{k=0}^\infty \frac{(-1)^k}{(2\omega)^k}\frac{\Gamma(\vartheta+k+\frac{1}{2})}{k!\Gamma(\vartheta-k+\frac{1}{2})}\;,\hspace{0.5cm} |\mathrm{Arg}(\omega)|<\frac{\pi}{2}\;.
\end{equation}
From Eqs. (\ref{inner product coherent states}) and (\ref{asymptotic expression modified Bessel function}) we conclude the proof of this proposition.
\end{proof}

\begin{remark} 
We are mainly interested in using proposition \ref{proposition inner product coherent states} for the cases $\vecz=\vecw$ (and then $\mathrm{Arg}(\vecz\cdot\vecw) = 0$) in this paper. The case when $| \mathrm{Arg}(\vecz\cdot\vecw)| = \pi$ requires the use of an asymptotic expression valid in a different region than the one we are considering in proposition \ref{proposition inner product coherent states}. Thus if we take the branch of the square root function given by $\sqrt z = |z|^{1/2} exp(\imath\mathrm{Arg(z)}/2)$ with $0 < \mathrm{Arg}(z) < 2\pi$ then by using formula 8.451-5 in Ref. \cite{G94} we obtain the asymptotic expression,
\begin{align}
 \left\langle \K(\cdot,\vecz),\K(\cdot,\vecw)\right\rangle _{_{\Sn}}= & \frac{\Gamma(n+p)}{2\sqrt \pi} \left(\frac{\vecw \cdot \vecz}{\hbar^2}\right)^{\frac{1}{2}(-p-n+\frac{1}{2})} \nonumber\\
 & \left[\mathrm{e}^{\frac{2}{\hbar}\sqrt{\vecw \cdot \vecz}}  + \mathrm{e}^{-\frac{2}{\hbar}\sqrt{\vecw \cdot \vecz}}\mathrm{e}^{\pi\imath (n+p-\frac{1}{2})}\right] [1 + \mathrm O(\hbar)]\;.\label{other expression inner product}
\end{align}
Note that both asymptotic expressions in Eqs. (\ref{asymptotic inner product coherent states}) and (\ref{other expression inner product}) coincide up to an error of the order $\mathrm O(\hbar^\infty)$ in the common region where they are valid.
\end{remark}

From the proposition \ref{proposition inner product coherent states}, we obtain an estimate for the norm of a given coherent states
\begin{prop}
For $\vecz \in \mathbb C^n-\{0\}$
\begin{align}
|| \K(\cdot,\vecz)||_{\Sn}^2 = \frac{\Gamma(n+p)}{2\sqrt \pi} &\left(\frac{|\vecz|}{\hbar}\right)^{-p-n+\frac{1}{2}} \mathrm{e}^{\frac{2}{\hbar}|\vecz|}\nonumber\\
& \left[1 -\frac{(n+p-\frac{3}{2})(n+p-\frac{1}{2})}{4|\vecz|}\hbar+\mathrm O(\hbar^2)\right]\;.\label{norm coherent states}
\end{align}
\end{prop}

We end this section showing that the family of coherent states $\mathcal K$ is complete
\begin{prop}\label{supercomplete}
The family $\{\K(\cdot,\vecz)\;|\;\vecz \in \mathbb C^n\}$ form a complete system for $\mathcal O$.
\end{prop}
 
\begin{proof} Let $\Phi,\Psi \in \mathcal O$, since the transform $\U$ is a unitary transformation   
we have
\begin{align*}
\langle \Phi,\Psi\rangle _{{\Sn}} & = (\U\Phi,\U\Psi)_p\\
& = \int_{\mathbb C^n} \U\Phi(\vecz)\overline{\U\Psi(\vecz)} \dm(\vecz)\\
& = \int_{\mathbb C^n} \langle \Phi,\K(\cdot,\vecz)\rangle _{{\Sn}} \langle \K(\cdot,\vecz),\Psi) \rangle _{{\Sn}}\dm(\vecz)\;,
\end{align*}
 where we have used Eq. (\ref{U inner product coherent states}).
 \end{proof}

\section{Berezin symbolic calculus}\label{section Berezin symbolic calculus}
According to Berezin's theory (see Ref. \cite{B74}), from Theorem \ref{unitary}, Eq. (\ref{U inner product coherent states}) and the proposition \ref{supercomplete}, we may consider the following
\begin{definition}
The Berezin's symbol of a bounded linear operator $A$ with domain in $\mathcal O$ is defined, for every $\vecz \in \mathbb C^n$, by
\begin{equation}\label{definition Berezin transform}
\B( A)(\vecz)=\frac{\langle A\K(\cdot,\vecz),\K(\cdot,\vecz)\rangle_{{\Sn}}}{\langle\K(\cdot,\vecz),\K(\cdot,\vecz)\rangle_{{\Sn}}}\;.
\end{equation}
\end{definition}
From Eq. (\ref{inner product coherent states}), we have that 
\begin{equation*}
||\K(\cdot,\vecz)||_{\Sn}^2=\Gamma(n+p) \left(\frac{|\vecz|}{\hbar}\right)^{-p-n+1}\mathrm I_{n+p-1}\left(\frac{2|\vecz|}{\hbar}\right)>0,
\end{equation*}
 hence the coherent states are continuous, i.e. the map $\vecz \to |\K(\cdot,\vecz)|$ is continuous. Therefore, if $A:\mathcal O\to \mathcal O$ is a bounded operator, its Berezin's symbol, can be extended uniquely to a function defined on a neighbourhood of the diagonal in $\mathbb C^n \times \mathbb C^n$  in such a way that it is holomorphic in the first factor and anti-holomorphic in the second. In fact, such an extension is given explicitly by
\begin{equation}\label{extended covariant symbol}
\B(A)(\vecz,\vecw):=\frac{\langle A\K(\cdot,\vecz),\K(\cdot,\vecw)\rangle_{{\Sn}}}{\langle\K(\cdot,\vecz),\K(\cdot,\vecw)\rangle_{{\Sn}}}\;.
\end{equation}
\begin{remark}
By Eqs. (\ref{inner product coherent states}), the extended Berezin's symbol has singularities in the zeros of the modified Bessel function $\mathrm I_{n+p-1}(z)$, which are well known (see Ref. \cite{L65} Sec. 5.13) and where $\vecw\cdot\vecz=0$. Then the extended Berezin's symbol is defined on $\mathbb C^n \times \mathbb C^n \setminus \mathcal S$, with
\begin{equation*}
\mathcal S=\left\{(\vecw,\vecz)\in \mathbb C^n \times \mathbb C^n\;|\; \vecw \cdot\vecz=0 \mbox{ or }2\sqrt{\vecw\cdot\vecz}/\hbar=\imath \lambda\right\}
\end{equation*}
where $\lambda$ is a negative real  number that satisfies  $\mathrm I_{n+p-1}(\imath\lambda)=0$.
\end{remark}

We now give the rules for symbolic calculus

\begin{prop}\label{proposition rules for Berezin}
Let $A, B$ bounded linear operators with domain in $\mathcal O$. Then for $\vecz,\vecw \in \mathbb C^n$ and $\phi\in \mathcal O$ we have
\begin{align}
\B(\mathrm{Id})& = 1\;, \mbox{with $\mathrm{Id}$ the identity operator,}\nonumber\\
\B(A^*) (\vecz,\vecw) & = \overline{\B(A)(\vecw,\vecz)}\;,\nonumber\\
\B(AB)(\vecz,\vecw) & = \frac{2}{(\pi\hbar)^n\hbar}\int\limits_{\mathbb C^n} \B(B)(\vecz,\vecu) \B(A)(\vecu,\vecw) \left[\frac{\vecu\cdot \vecz\; \vecw\cdot \vecu}{\vecw\cdot\vecz}\right]^{\frac{1}{2}(-p-n+1)}\nonumber\\
 & \hspace{0.5cm} \frac{\mathrm I_{n+p-1}\left(\frac{2\sqrt{\vecu\cdot \vecz}}{\hbar}\right)\mathrm I_{n+p-1}\left(\frac{2\sqrt{\vecw\cdot \vecu}}{\hbar}\right)}{\mathrm I_{n+p-1}\left(\frac{2\sqrt{\vecw\cdot \vecz}}{\hbar}\right)} |\vecu|^{p} \mathrm K_{p}\left(\frac{2|\vecu|}{\hbar}\right) \mathrm d\vecu \mathrm d\overline\vecu\;,\label{Eq 3}\\
 \U(A\phi)(\vecz) & = \frac{2}{(\pi\hbar)^n\hbar}\int\limits_{\mathbb C^n} \U \phi(\vecw) \B(A)(\vecw,\vecz)  \big(\vecz \cdot \vecw\big)^{\frac{1}{2}(-p-n+1)}  \nonumber\\
 &\hspace{0.5cm}\mathrm I_{n+p-1}\left(\frac{2\sqrt{\vecz\cdot \vecw}}{\hbar}\right) |\vecw|^{p} \mathrm K_{p}\left(\frac{2|\vecw|}{\hbar}\right) \mathrm d\vecw \mathrm d\overline\vecw.\label{Eq 4}
\end{align}
\end{prop}
\begin{proof}
This is a direct consequence from the formulas in Ref. \cite{B74} (see Eqs. (\ref{Eq 1}),  (\ref{Eq 2})), and the Eqs. (\ref{U inner product coherent states}) and (\ref{inner product coherent states}).
\end{proof}

\begin{corollary}
Let $A:L^2(\Sn)\to L^2(\Sn)$ be a bounded operator, and $A^\#$ the operator on $\E$ with Schwartz kernel the function $\mathcal K_A$ defined on $\mathbb C^n \times \mathbb C^n$, by 
\begin{displaymath}
\mathcal K_A(\vecz,\vecw):=\langle A\K(\cdot,\vecw),\K(\cdot,\vecz)\rangle_{_{\Sn}}\;.
\end{displaymath}
Then
\begin{equation}
A^\#f(\vecz):=\U A(\U)^{-1} f(\vecz)\;\hspace{0.5cm}\forall f\in \E\;,\;\vecz \in \mathbb C^n.
\end{equation}
\end{corollary}
\begin{proof} Let $f \in \E$, and $\phi=(\U)^{-1} f$. From Eqs. (\ref{Eq 4}),  (\ref{extended covariant symbol}) and (\ref{inner product coherent states})
\begin{align*}
\U A (\U)^{-1} f (\vecz)= \int\limits_{\mathbb C^n} f(\vecw) \mathcal K_A(\vecz,\vecw) \dm(\vecw)\,.
\end{align*}
\end{proof}

\subsection{Asymptotic expansion of the Berezin's symbol.}
In this section, we obtain asymptotic  expansion of Berezin's symbol of Toeplitz operator.


Let $\Pro:L^2(\Sn) \to \mathcal O$ the orthogonal projection and $A$ pseudo-differential operator  on $\Sn$ of order zero. The \textsl{Toeplitz operator} is defined as
\begin{displaymath}
\To_A = \Pro A \Pro\;.
\end{displaymath}
For $\psi \in C^\infty(\Sn)$, with $C^\infty(\Sn)$ be the algebra of complex-valued $C^\infty$ functions on $\Sn$, let $\mathbf M_\psi$ the operator of multiplication by $\psi$. For simplicity we use notation $\To_\phi:=\To_{\mathrm M_\phi}$.


\begin{theorem}
Let  $\bk\in \mathbb Z_+^n$ be a multi-index. Then
\begin{align*}
\B(\To_{\vecx^\bk})(\vecz,\vecw) & = \frac{\left(\displaystyle\frac{\vecw\cdot\vecz}{\hbar^2}\right)^{\frac{1}{2}(n+p-1)}}{\mathrm I_{n+p-1}\left(2\sqrt{\vecw\cdot\vecz}/\hbar\right)}\left(\frac{\vecw}{\hbar}\right)^\bk  \sum_{\ell=0}^\infty \left(\frac{\vecw\cdot\vecz}{\hbar^2}\right)^\ell \frac{c_{\ell,p}}{c_{|\bk|+\ell,p}}\\
&\hspace{0.5cm} \frac{1}{\ell!\Gamma(|\bk|+\ell+n+p)}
\end{align*}
where $c_{\ell,p}$ was define in Eq. (\ref{constants coherent states}).
\end{theorem}
\begin{proof} From Eq. (\ref{extended covariant symbol}) and properties the orthogonal projection
\begin{align*}
\B(\To_{\vecx^\bk})(\vecz,\vecw) = \frac{\langle \vecx^{\bk}\K(\cdot,\vecz),\K(\cdot,\vecz)\rangle_{{\Sn}}}{\langle\K(\cdot,\vecz),\K(\cdot,\vecz)\rangle_{{\Sn}}}\;.
\end{align*}
Using the dominated convergence theorem, the Lemma \ref{product in Sn} and  Eq. (\ref{inner product coherent states}) we conclude the proof of this theorem.
\end{proof}

\begin{remark}
In the particular case when $p=0$, we obtain from the last theorem
\begin{align*}
\mathfrak B_{\hbar,0}(\To_{\vecx^\bk})(\vecz,\vecw) & = \left(\frac{\vecw}{\sqrt{\vecw\cdot\vecz}}\right)^\bk \frac{\mathrm I_{n+|\bk|-1}\left(2\sqrt{\vecw\cdot\vecz}/\hbar\right)}{\mathrm I_{n-1}\left(2\sqrt{\vecw\cdot\vecz}/\hbar\right)}.
\end{align*}
\end{remark}


\begin{theorem}\label{theorem asymptotic Berezin transform}
Let $p=0$, for any $\vecz \ne 0$ and $\Phi$ a smooth function on $\Sn$, the Berezin symbol $\mathfrak B_{\hbar,0}$ associated to the Toeplitz operator $\To_\Phi$ has the asymptotic expansion
\begin{align}
\mathfrak B_{\hbar,0}\To_\Phi(\vecz)& =\Phi(\vecz/|\vecz|)+\mathrm O(\hbar)\;.
\label{asymptotic Berezin transform}
\end{align}
\end{theorem}

\begin{proof}
First we observe from Eq. (\ref{coherent states}), since $p=0$, that $\K(\vecx,\vecz)=e^{\vecx\cdot\vecz/\hbar}$. Let $\Phi\in C^\infty(\Sn)$, from Eq. (\ref{norm coherent states}) we have
\begin{align}
\mathfrak B_{\hbar,0} \To_\Phi (\vecz) & =\frac{2\sqrt \pi}{\Gamma(n)} \left(\frac{|\vecz|}{\hbar}\right)^{n-\frac{1}{2}}\int_{\Sn} \mathrm{exp}\left(\frac{2}{\hbar} \big[\Re (\vecx\cdot\vecz)-|\vecz|\big]\right) \Phi(\vecx)\nonumber\\[0.5cm] 
& \hspace{3cm}\left(1+ \mathrm{O}(\hbar)\right) \ds(\vecx)\;.\label{Berezin approximation 1}
\end{align}

We identify $\mathbb R^{2n}$ with $\mathbb C^n$ in the usual way: for $(x_1,\ldots,x_n)\in \Sn$, let $\vecy=(y_1,\ldots,y_{2n})$ with $x_j=y_{j}+\imath y_{n+j}$. Note that $\vecy \in S^{2n-1}=\{\mathbf a \in \mathbb R^{2n}\;|\; |\mathbf a|=1\}$. Then,  the argument of the exponential function in Eq. (\ref{Berezin approximation 1}) is 
\begin{displaymath}
\frac{2}{\hbar}\;\big[\vecy \cdot(\Re(\vecz),\Im(\vecz))-|\vecz|\big]\;.
\end{displaymath}

In order to estimate (\ref{asymptotic Berezin transform}), we define
\begin{align}
\mathcal A &: = \left(\frac{|\vecz|}{\pi\hbar}\right)^{n-\frac{1}{2}} \int\limits_{S^{2n-1}}  \Psi(\vecy)\mathrm{exp}\left(\frac{2}{\hbar} \big[\vecy\cdot(\Re(\vecz),\Im(\vecz))-|\vecz|\big]\right) \mathrm d\Omega(\vecy)\label{Auxiliar A}
\end{align}
where $\mathrm d\Omega$ is the surface measure on $S^{2n-1}$ and $\Psi$ is a smooth function on $S^{2n-1}$.

Note that given $(\Re(\vecz),\Im(\vecz))\in \mathbb R^{2n}$, there exist a rotation $R$ in $\mathrm {SO}(2n)$ such that $(\Re(\vecz),\Im(\vecz))= rR \hat{\mathbf e}_1 $ with $r=|\vecz|$ and $\hat{\mathbf e}_1$ is canonical unit vector in $\mathbb R^{2n}$. Thus we have
\begin{align}
\mathcal A& = \left(\frac{r}{\pi\hbar}\right)^{n-\frac{1}{2}} \int\limits_{\boldsymbol \omega\in S^{2n-1}}  \Psi(R\boldsymbol \omega) \mathrm{exp}\left(\frac{\imath}{\hbar}f_{\vecz}(\boldsymbol \omega)\right)  \mathrm d\Omega(\boldsymbol \omega) \label{Auxiliar A-1}
\end{align}
where $f_{\vecz}(\boldsymbol \omega)=-2\imath r\big[\omega_1-1 \big]$.

Let us introduce spherical coordinates for the variables $(\omega_1,\ldots, \omega_{2n})\in S^{2n-1}$:
\begin{align*}
\omega_1 & = \sin(\theta_{2n-1}) \cdots \sin(\theta_2)\sin(\theta_1)\;,\\
\omega_2 & = \sin(\theta_{2n-1}) \cdots \sin(\theta_2)\cos(\theta_1)\;,\\
&\hspace{0.2cm}\vdots\\
\omega_{2n-1}&=\sin(\theta_{2n-1})\cos(\theta_{2n-2})\;,\\
\omega_{2n}&= \cos(\theta_{2n-1})
\end{align*}
with $\theta_1\in(-\pi,\pi)$, $\theta_2,\theta_3,\cdots,\theta_{2n-1}\in(0,\pi)$.  

The function $f_\vecz$ appearing in the argument of the exponential function in Eq. (\ref{Auxiliar A-1}) has a non-negative imaginary part and has only one critical point  (as a function of the angles) $\boldsymbol \theta_0$ which contributes to the asymptotic given by $\theta_j=\pi/2$, $j=1,\ldots,2n-1$. In addition, since
\begin{equation*}
\left.\frac{\partial^2 f_\vecz}{\partial \theta_\ell\partial \theta_j}\right|_{\vtheta=\boldsymbol \theta_0}=2\imath r \delta_{j\ell}\;,
\end{equation*}
with $\delta_{ij}$ denoting the Kronecker symbol, then the determinate of the Hessian matrix of the function $f_\vecz$ evaluated at the critical point is equal to $(2\imath r)^{2n-1}$.

From the stationary phase method (see Ref. \cite{H90}) we obtain
\begin{align}
\mathcal A & = \left[\Psi(R\boldsymbol \omega(\vtheta))\right]_{\vtheta=\vtheta_0}+  \mathrm O(\hbar)=\Psi((\Re(\vecz),\Im(\vecz))/|\vecz|)+\mathrm O(\hbar)\;.\label{auxiliar mathcal A}
\end{align}

For $\vecx \in \Sn$, we consider $\Psi(\Re(\vecx),\Im(\vecx))=\Phi(\vecx)$. Then from Eqs. (\ref{Berezin approximation 1}), (\ref{Auxiliar A}) and (\ref{auxiliar mathcal A}) we conclude the proof of this theorem.
\end{proof}

When $p = -1$ we can give an asymptotic expression of the coherent states (see Appendix \ref{AppendixB}). This result will give us as a consequence an asymptotic expression for Berezin symbol $\mathfrak B_{\hbar,-1}$ to the Toeplitz operator. 
\begin{theorem}\label{theorem asymptotic Berezin transform p=1}
Let $p=-1$, for any $\vecz \ne 0$ and $\Phi$ a smooth function on $\Sn$, the Berezin transform $\mathfrak B_{\hbar,-1}$ associated to the Toeplitz operator $\To_\Phi$ has the asymptotic expansion
\begin{align}
\mathfrak B_{\hbar,-1}\To_\Phi(\vecz)& =\Phi(\vecz/|\vecz|)+\mathrm O(\hbar)\;.\label{asymptotic Berezin transform p=1}
\end{align}
\end{theorem}
\begin{proof}
Let us define the following regions on $\Sn$, with the constant $C$ mentioned in the Lemma \ref{asymptotic p=1} taken greater than one
\begin{equation}
W=\left\{\vecx \in \Sn\;|\; C\frac{\Re(\vecx\cdot\vecz)}{|\vecz|}\ge 1\right\}\;\;\mbox{and}\;\; V=\Sn-W\;.
\end{equation}
Note that $\vecx \in W$, implies $|\Im(\vecx\cdot\vecz)|\le |\vecz|\le C\Re(\vecx\cdot\vecz)$ and therefore we can use the asymptotic expression of coherent states (see proposition \ref{asymptotic coherent states p=1}).

Let 
\begin{equation}
A=\int\limits_{\vecx\in W} \Phi(\vecx) \frac{|\Ku (\vecx,\vecz)|^2}{||\Ku(\cdot,\vecz)||^2_{\Sn}} \ds(\vecx)\; ,\hspace{0.2cm} B=\int\limits_{\vecx\in V} \Phi(\vecx) \frac{|\Ku (\vecx,\vecz)|^2}{||\Ku(\cdot,\vecz)||^2_{\Sn}} \ds(\vecx).
\end{equation}
We affirm that the integral on V is actually $\mathrm O(h^\infty)$. To try this, from Lemma 10.1 of Ref \cite{D-V09} (specifically Eq. 10.4), the function $\Ku(\cdot,\vecz)$ has a integral expression given by
\begin{align}
\Ku(\vecx,\vecz) & = \frac{2\sqrt{n-1}}{\pi}e^{\vecx\cdot\vecz/\hbar}\int_0^1\left[\frac{\vecx\cdot\vecz}{(n-1)\hbar}(1-\omega^2)+1\right]e^{-\omega^2\vecx\cdot\vecz/\hbar}\nonumber\\
& \hspace{2cm} \frac{(1-\omega^2)^{n-2}\omega}{\sqrt{-\ln(1-\omega^2)}}\;\mathrm d\omega.\nonumber
\end{align}
From the last equation we obtain
\begin{align*}
|\Ku(\vecx,\vecz)| & \le C_1e^{\frac{|\vecz|}{\hbar}}\int\limits_0^1\left[\frac{|\vecz|}{(n-1)\hbar}+1\right]\mathrm {exp}\left(\frac{|\vecz|}{\hbar}\left[\frac{\Re(\vecx\cdot\vecz)}{|\vecz|}(1-\omega^2)-1\right]\right)\;\mathrm d\omega\nonumber
\end{align*}
for some constant $C_1$. Note that, for $\vecx \in V$
\begin{equation*}
\frac{\Re(\vecx\cdot\vecz)}{|\vecz|}(1-\omega^2)-1\le \mu \;,\mbox{ with } \mu=\frac{1}{C}-1<0\;.
\end{equation*}
Thus we get the estimate for $\vecx\in V$
\begin{align*}
|\Ku(\vecx,\vecz)| & \le C_1e^{\frac{|\vecz|}{\hbar}}e^{\mu\frac{|\vecz|}{\hbar}}\left[\frac{|\vecz|}{(n-1)\hbar}+1\right].
\end{align*}
From the norm estimate for the coherent states in Eq. (\ref{norm coherent states}) we obtain
\begin{equation*}
\frac{|\Ku(\vecx,\vecz)|^2}{||\Ku(\cdot,\vecz)||^2_{\Sn}} \le C_1 \left(\frac{|\vecz|}{\hbar}\right)^{n+\frac{3}{2}} e^{2\mu\frac{|\vecz|}{\hbar}}\left[\frac{|\vecz|}{(n-1)\hbar}+1\right]^2(1+\mathrm O(\hbar))\;.
\end{equation*}
Hence $B=\mathrm{O}(\hbar^\infty)$. On the other hand, from proposition \ref{asymptotic coherent states p=1} and Eq. (\ref{norm coherent states}) we have
\begin{align}
A & =\frac{2\sqrt \pi}{\Gamma(n)} \left(\frac{|\vecz|}{\hbar}\right)^{n-\frac{1}{2}}\int\limits_W  \frac{|\vecx\cdot\vecz|}{|\vecz|}\Phi(\vecx) \mathrm{exp}\left(\frac{2|\vecz|}{\hbar}\left[\frac{\Re(\vecx\cdot\vecz)}{|\vecz|}-1\right]\right)\nonumber\\
&\hspace{2cm}\left[1+\mathrm O(\hbar)\right]\ds\;. \label{A p=1}
\end{align}
Note that, for $\vecx \in V$, $\frac{\Re(\vecx\cdot\vecz)}{|\vecz|}-1<\frac{1}{C}-1<0$.  Thus, we can take the integral defining $A$ in Eq. (\ref{A p=1}) over the whole sphere with error $\mathrm O(\hbar^\infty)$.
From Eq. (\ref{Auxiliar A}), considering 
\begin{equation}
 \Psi(\Re(\vecx),\Im(\vecx))=\Phi(\vecx)\frac{|\vecx\cdot\vecz|}{|\vecz|}\;, \hspace{0.5cm} \vecx \in \Sn
\end{equation} 
 and (\ref{auxiliar mathcal A}) we conclude the proof of theorem.
\end{proof}


\section{The star product}\label{section star product}
In Ref. \cite{B74}, Berezin show that the formula (\ref{Eq 3}) will allow us to define a start product (see Refs. \cite{BF78} for the standard definition of star-product) on  the algebra $\A$ out of Berezin's symbol for bounded operators with domain in $\mathcal O$ (see Eq. (\ref{definition Berezin transform})), i.e.
\begin{displaymath}
\A=\{\B(A)\;|\; A\in \mathbf B (\mathcal O)\}.
\end{displaymath}
In this section we verifies that this noncommutative star-product, which will be denoted by $*_p$, satisfies the usual requirement on the semiclassical limit, i.e. as $\hbar\to 0$
\begin{equation*}
f_1*_p f_2(\vecz)=f_1(\vecz)f_2(\vecz)+\hbar B(f_1,f_2)(\vecz)+\mathrm O(\hbar^2),\hspace{0.5cm}\vecz\in \mathbb C^n, f_1,f_2\in \A, 
\end{equation*}
where $B(\cdot,\cdot)$ is a certain bidifferential operator of first order.

\begin{theorem}\label{theorem asymptotic expansion} Let $n\ge 2$, $\vecz\in \mathbb C^n$, $\beta$ a smooth function defined on $\mathbb C^n$, $p> -n$ and $\mu,\nu \in \mathbb R$ with $\mu,\nu>\frac{1}{2}$. Then for $\hbar \to 0$
\begin{align}
\mathbf I (\vecz) &: = \int_{\mathbb C^n} \beta(\vecw) \frac{\F(\nu,\vecw\cdot\vecz/\hbar^2) \F(\mu,\vecz\cdot\vecw/\hbar^2)}{\F(\mu,|\vecz|^2/\hbar^2)}\dm(\vecw)\nonumber\\
& = \left(\frac{|\vecz|}{\hbar}\right)^{n+p-\nu}\frac{\Gamma(\nu)}{\Gamma(n+p)}
\left\{ \beta(\vecz) + \hbar\left(\frac{1}{4|\vecz|} (p-\nu+1)(p+\nu-1)\beta(\vecz)\right.\right.\nonumber\\ 
&\hspace{1cm}+ \frac{R}{2}\beta(\vecz)+g^{\bar j i}(\vecz)\left[\frac{1}{2|\vecz|^2}\big[\overline z_i(n+p-\nu)\partial_{\bar j} + z_j(n+p-\mu)\partial_{i}\big]\beta(\vecw)\right.\nonumber\\
& \hspace{1cm}\left. \left.\left.+\partial_{i}\partial_{\bar j} \beta(\vecw)+\frac{\beta(\vecz)}{|\vecz|^{n+p-\nu}} \partial_{\bar j}\partial_i \left(\frac{\xi_\vecz^{\nu,\mu}}{g}(\vecw)\right)\right]_{\vecw=\vecz} \right)\right\} + \mathrm O(\hbar^2) \label{asymptotic expansion} 
\end{align}
 with $\F$ denoting the generalized hypergeometric function (see Sec.9.14 of Ref. \cite{L65} for definition and expressions for this special function),  $\partial_i=\partial/\partial w_i$ and $\partial_{\bar j}=\partial/\partial \overline w_j$. Further, where $R$ is the scalar curvature defined by the K\"ahler metric $\mathrm d s^2 =g_{i  \bar j} \mathrm d \overline z_j \wedge \mathrm d z_i$, and $g_{i\bar j}, g^{i\bar j}, \xi_\vecz^{\nu,\mu}$, $g$ are functions described subsequently.
\end{theorem}

\begin{proof} 
Let us define the following two regions on $\mathbb C^n$
\begin{equation}
W=\left\{\vecw \in \mathbb C^n\;|\; \frac{\Re\sqrt{\vecw\cdot\vecz}}{|\vecz|}\ge \frac{1}{4}\right\}\;\;\mbox{and}\;\; V=\mathbb C^n-W\;.
\end{equation}
The integral in Eq. (\ref{asymptotic expansion}) can be written as an integral on the region $W$ plus an integral on $V$ denoted by the letters $\mathbf J_W$ and $\mathbf J_V$ respectively. The integral on $W$ is the one giving us the main asymptotic and the integral on $V$ is actually $\mathrm O(\hbar^\infty)$. In order to estimate $\mathbf J_V$, from the equality 
\begin{equation}\label{equality bessel = hypergeometric}
\Gamma(\nu +1)\left(\frac{1}{2}z\right)^{-\nu}\mathrm I_\nu(z)=\F(\nu +1,z^2/4)
\end{equation}
 (see formula 9.6.47 of Ref. \cite{A-S72} for details), and the integral representation of the function $\mathrm I_\lambda$ (see formula 8.431-1 of Ref. \cite{G94})
\begin{equation*}
\mathrm I_\lambda(z)=\frac{\left(\frac{z}{2}\right)^\lambda}{\Gamma\left(\lambda+\frac{1}{2}\right)\Gamma\left(\frac{1}{2}\right)}\int_{-1}^1(1-t^2)^{\lambda-\frac{1}{2}}e^{\pm zt}\mathrm dt\;,\hspace{0.5cm} \Re\left(\lambda\right) > -\frac{1}{2}
\end{equation*}
we have
\begin{align*}
\left|\F\left(\nu,\frac{\vecw\cdot\vecz}{\hbar}\right)\right| &  = C \left|\int_{-1}^1 (1-t^2)^{\nu-\frac{3}{2}}e^{\frac{2}{\hbar}\sqrt{\vecw\cdot\vecz}\;t} \mathrm dt\right|\\
& \le C e^{\frac{1}{\hbar}|\vecz|} \int_{-1}^1(1-t^2)^{\nu-\frac{3}{2}}\mathrm{exp}\left({\frac{|\vecz|}{\hbar}\left[2\frac{\Re\sqrt{\vecw\cdot\vecz}}{|\vecz|}-1\right]}\right)\mathrm dt\:.
\end{align*}
for some constant $C$. Note that, for $\vecw \in V$,
\begin{equation}\label{points in V}
2\frac{\Re\sqrt{\vecw\cdot\vecz}}{|\vecz|}-1<-\frac{1}{2}\;.
\end{equation}
Thus we get the estimate for $\vecw\in V$
\begin{align*}
\left|\F\left(\nu,\frac{\vecw\cdot\vecz}{\hbar}\right)\right| & \le C e^{\frac{1}{\hbar}|\vecz|} e^{-\frac{1}{2\hbar}|\vecz|}\int_{-1}^1(1-t^2)^{\nu-\frac{3}{2}}\mathrm dt\\
& = C \mathrm B\left(\frac{1}{2},\nu-\frac{1}{2}\right) e^{\frac{1}{\hbar}|\vecz|} e^{-\frac{1}{2\hbar}|\vecz|}\;,
\end{align*}
where $\mathrm B(x,y)$ denoting the beta function (see 8.38 of Ref. \cite{G94} for definition and expression for this special function). From the equality $\F(\mu,|\vecz|^/\hbar^2)=\Gamma(\mu) |\vecz|^{-\mu+1}\mathrm I_{\mu-1}(2|\vecz|/\hbar)$ (see Eq. (\ref{equality bessel = hypergeometric})) and Eq. (\ref{asymptotic expression modified Bessel function}) we obtain for $\vecw\in V$
\begin{equation}
 \frac{|\F(\nu,\vecw\cdot\vecz/\hbar^2) |\;|\F(\mu,\vecz\cdot\vecw/\hbar^2)|}{\F(\mu,|\vecz|^2/\hbar^2)} \le C\left(\frac{|\vecz|}{\hbar}\right)^{\mu-\frac{1}{2}} e^{-\frac{1}{\hbar}|\vecz|}\left(1+\mathrm O(\hbar)\right)\;.
\end{equation}
Hence $\mathbf J_V=\mathrm O(\hbar^\infty)$. Let us now study the term $\mathbf J_W$, 
first we note that for $\vecw\in W$, $\mathrm{Arg}(\vecw\cdot\vecz)<\pi$. So from Eqs. (\ref{equality bessel = hypergeometric}), (\ref{asymptotic expression modified Bessel function}), the expression for the measure $\dm $ (see Eq. (\ref{measure on U})) and the asymptotic expression of the MacDonal-Bessel function of order $p$, $\mathrm K_p$(see formula  9.7.2 of Ref. \cite{A-S72}) we find that
\begin{align*}
\mathbf J_W (\vecz)  &=\int_{W} \beta(\vecw) \frac{\F(\nu,\vecw\cdot\vecz/\hbar^2) \F(\mu,\vecz\cdot\vecw/\hbar^2)}{\F(\mu,|\vecz|^2/\hbar^2)}\dm(\vecw)\nonumber\\
& = \frac{1}{2\pi^2}\frac{\hbar^\nu}{\hbar^{2n+p}}\frac{\Gamma(\nu)}{\Gamma(n+p)}\int_{W}\beta(\vecw) \xi_{\vecz}^{\nu,\mu}(\vecw) \chi_\vecz(\vecw)\mathrm{exp} \left(\frac{1}{\hbar} f_\vecz(\vecw)\right) \mathrm d\vecw \mathrm d \overline\vecw
\end{align*}
with
\begin{align}
\xi_{\vecz}^{\nu,\mu}(\vecw) & = |\vecz|^{\mu-\frac{1}{2}} |\vecw|^{p-\frac{1}{2}}(\vecz\cdot\vecw)^{-\frac{\mu}{2}+\frac{1}{4}}(\vecw\cdot\vecz)^{-\frac{\nu}{2}+\frac{1}{4}},\nonumber\\ 
\chi_\vecz(\vecw) & = 1 + \frac{\hbar}{16}\left[\frac{1-4(\mu-1)^2}{\sqrt{\vecz\cdot\vecw}} + \frac{1-4(\nu-1)^2}{\sqrt{\vecw\cdot\vecz}} +\frac{4p^2-1}{|\vecw|} \right.\nonumber\\
& \hspace{0.5cm}\left.+ \frac{4(\mu-1)^2-1}{|\vecz|})\right] + E(\vecz,\hbar)\;,\label{error}\\
f_\vecz(\vecw) & = 2\left(\sqrt{\vecz\cdot \vecw}+\sqrt{\vecw\cdot\vecz}-|\vecw|-|\vecz|\right) =2\left(2\Re(\sqrt{\vecz\cdot\vecw})-|\vecz|-|\vecw|\right),\nonumber
\end{align}
where the error term $E(\vecz,\hbar)$ in (\ref{error}) is $\mathrm O(\hbar^2)$ uniformly with respect to $\vecw \in W$ because in such a region we have $\frac{1}{|\sqrt{\vecw\cdot\vecz}|}\le\frac{4}{|\vecz|}$ and $\frac{1}{|\vecw|}\le \frac{4^4}{|\vecz|^3}$. 

Furthermore, from Eq. (\ref{points in V}) $f_\vecz(\vecw)\le -|\vecz| -|\vecw|<0$ for $\vecw\in V$. Thus we can take the integral defining $\mathbf J_W$ over the whole space $\mathbb C^n$ with an error $\mathrm O(\hbar^\infty)$.

Let us now write $ \mathbf J_W (\vecz)=\mathrm B(\vecz)+\hbar\; \mathrm C(\vecz) + \mathrm O(\hbar^2)$ with
\begin{align}
\mathrm B(\vecz) & = \frac{1}{2\pi^2}\frac{\hbar^\nu}{\hbar^{2n+p}}\frac{\Gamma(\nu)}{\Gamma(n+p)} \int\limits_{\mathbb C^n} \beta(\vecw) \xi_\vecz^{\nu,\mu}(\vecw) \mathrm{exp}{ \left(\frac{1}{\hbar} f_\vecz(\vecw)\right)} \mathrm d\vecw \mathrm d \overline\vecw,\label{B}\\
\intertext{and}
\mathrm C(\vecz)& = \frac{1}{32\pi^2}\frac{\hbar^\nu}{\hbar^{2n+p}}\frac{\Gamma(\nu)}{\Gamma(n+p)} \int\limits_{\mathbb C^n} \beta(\vecw)\xi_\vecz^{\nu,\mu}(\vecw) \left[
\frac{1-4(\mu-1)^2}{\sqrt{\vecz\cdot\vecw}} + \frac{1-4(\nu-1)^2}{\sqrt{\vecw\cdot\vecz}}\right.\nonumber\\
&\hspace{1cm}\left.+\frac{4p^2-1}{|\vecw|} + \frac{4(\mu-1)^2-1}{|\vecz|}) \right] \mathrm{exp} \left(\frac{1}{\hbar} f_\vecz(\vecw)\right) \mathrm d\vecw \mathrm d \overline\vecw.\label{C}
\end{align}

For obtain the asymptotic expansion (\ref{asymptotic expansion}), we express  the integrals $\mathrm B(\vecz)$ and $\mathrm C(\vecz)$, using the stationary phase method (see Ref \cite{H90}), as $\mathrm B_1(\vecz) + \hbar \mathrm B_2(\vecz) + \mathrm O(\hbar^2)$, $\mathrm C_1(\vecz)+\mathrm O(\hbar)$ respectively.


For our purpose, note that, using Schwartz inequality,  the phase function $f_\vecz$ satisfies $f_\vecz(\vecw)\le -2(\sqrt{|\vecz|}-\sqrt{|\vecw|})^2\le 0$. Moreover $f_\vecz(\vecw)=0$ if and only if $\vecw=\mathrm e^{\imath \theta}\vecz$ for some $\theta\in [0,2\pi)$ and $\Im(\sqrt{\vecz\cdot\vecw})=0$. Therefore, $f_\vecz(\vecw)=0$ if and only if $\vecw=\vecz$. Thus $f_\vecz$ is smooth function on a neighbourhood of the critical point $\vecw=\vecz$. 
We also have
\begin{displaymath}
-\partial_{\bar j}\partial_i f_\vecz(\vecw)=
\frac{1}{|\vecw|}\left(\delta_{i,j}-\frac{w_j\overline w_i}{2|\vecw|^2}\right).
\end{displaymath}
Using that  fact that for any complex numbers $a_i, b_i$, $i=1,\ldots,n$,  $\mathrm{det} (\delta_{i,j}\pm a_ib_j)=1\pm \sum_\ell a_\ell b_\ell$, we obtain 
\begin{equation}\label{Hessian f_z}
\mathrm{det}(-\partial_{\bar j}\partial_i f_\vecz(\vecw))=
\frac{1}{2|\vecw|^n}\;.
\end{equation}

From the stationary phase method, applied to the integral $\mathrm C(\vecz)$, we deduce
\begin{equation}\label{C1}
\mathrm C_1(\vecz)=\left(\frac{|\vecz|}{\hbar}\right)^{n+p-\nu}\frac{\Gamma(\nu)}{\Gamma(n+p)}\frac{1}{4|\vecz|} (p-\nu+1)(p+\nu-1)\beta(\vecz)\;.
\end{equation}

Furthermore, the computation of $\mathrm B_1(\vecz)$ and $\mathrm B_2(\vecz)$ in Eq. (\ref{B}) 
it is using theorem 3 in Ref. \cite{E00}. For this we consider the K\"{a}hler potential $\Phi(\vecw)=2|\vecw|$, and let $\Phi(\vecw,\vecu)=2\sqrt{\vecw\cdot\vecu}$ be a sesqui-analytic extension of $\Phi(\vecw)=\Phi(\vecw,\vecw)$  to a neighbourhood of the diagonal.

From it, we have
\begin{displaymath}
f_\vecz(\vecw)= \Phi(\vecz,\vecw)+\Phi(\vecw,\vecz)-\Phi(\vecw,\vecw)-\Phi(\vecz,\vecz)\;,
\end{displaymath}
and 
\begin{align*}
g(\vecw) & =\mathrm{det}\left[g_{i\overline j}(\vecw)\right],\\
g_{i\overline j} & =\partial_{\bar j}\partial_i \Phi=-\partial_{\bar j}\partial_i f_\vecz(\vecw).
\end{align*}

From theorem 3 in Ref. \cite{E00} and Eq. (\ref{Hessian f_z})
\begin{align}
\mathrm B_1(\vecz) & = \frac{1}{2\pi^n}\frac{\hbar^\nu}{\hbar^{2n+p}}\frac{\Gamma(\nu)}{\Gamma(n+p)}\left[\frac{(\pi\hbar)^n}{g(\vecz)}\right]|\vecz|^{n+p-\nu}\beta(\vecz)\nonumber\\
& =\left(\frac{|\vecz|}{\hbar}\right)^{n+p-\nu}\frac{\Gamma(\nu)}{\Gamma(n+p)}\beta(\vecz)\; \label{B1}
\end{align}
and 
\begin{align}
\mathrm B_2(\vecz) & = \frac{1}{2\pi^n}\frac{\hbar^\nu}{\hbar^{2n+p}}\frac{\Gamma(\nu)}{\Gamma(n+p)}(\pi \hbar)^n\left[\mathrm L_1\left(\left.\frac{\beta(\vecw)}{g(\vecw)}\xi_\vecz^{\nu,\mu}(\vecw)\right)\right|_{\vecw=\vecz}+\frac{R}{2}\frac{\beta(\vecz)}{g(\vecz)}|\vecz|^{p-\nu}\right]\nonumber\\
\intertext{where $\mathrm L_1:= g^{\bar j i} \partial_i\partial_{\bar j} $ is the Laplace-Beltrami operator, with $g^{i\bar j}$ the coefficients of the inverse matrix of $[g_{i\bar j}(\vecw)]$, and $R:=\mathrm L_1(\log g)$ is the scalar curvature. Then}
&  = \left(\frac{|\vecz|}{\hbar}\right)^{n+p-\nu}\frac{\Gamma(\nu)}{\Gamma(n+p)}\left\{g^{\bar j i}(\vecz)\left[\partial_{i}\partial_{\bar j} \beta(\vecw)+\frac{\beta(\vecz)}{|\vecz|^{n+p-\nu}} \partial_{\bar j}\partial_i \left(\frac{\xi_\vecz^{\nu,\mu}}{g}(\vecw)\right)\right.\right.\nonumber\\
& \hspace{0.5cm}\left. \left.+\frac{1}{2|\vecz|^2}\big[\overline z_i(n+p-\nu)\partial_{\bar j} + z_j(n+p-\mu)\partial_{i}\big]\beta(\vecw)\right]_{\vecw=\vecz}+ \frac{R}{2}\beta(\vecz)\right\}\label{B2}
\end{align}
where we have use the definition  $\mathrm L_1$ and 
\begin{align*}
\left.\frac{\partial}{\partial w_i}\frac{\xi_\vecz^{\nu,\mu}}{g}\right|_{\vecw=\vecz} & = |\vecz|^{n+p-\nu-2}(n+p-\nu)\overline z_i\\
\left.\frac{\partial}{\partial \overline w_j}\frac{\xi_\vecz^{\nu,\mu}}{g}\right|_{\vecw=\vecz} & = |\vecz|^{n+p-\nu-2}(n+p-\mu) z_j
\end{align*}
\end{proof}
We note that in Eq. (\ref{asymptotic expansion}) appears $\partial_{\bar j}\partial_i (\xi_\vecz^{\nu,\mu}/g)$, which is not difficult to prove that
\begin{align}
\left.\partial_{\bar j}\partial_i \frac{\xi_\vecz^{\nu,\mu}}{g}(\vecw)\right |_{\vecw=\vecz} & =\frac{|\vecz|^{n+p-\nu}}{|\vecz|^4}\left[\left(n+p-\frac{1}{2}\right)\big[-z_j\overline z_i+\delta_{ij}|\vecz|^2\big]\right.\nonumber\\ 
&\hspace{1cm}+z_j\overline z_i(n+p-\nu)(n+p-\mu)\biggl]\label{second derivative}
\end{align}
which results in later use.

We remember that for $f_1,f_2\in \A$, the law of multiplication in $\A$ is
\begin{align}
\big(f_1 *_p f_2\big)(\vecz) & =\frac{2}{(\pi\hbar)^n\hbar}\int\limits_{\mathbb C^n} f_1(\vecu,\vecz) f_2(\vecz,\vecu) \left[\frac{|\vecu\cdot \vecz|}{|\vecz|}\right]^{-p-n+1}\nonumber\\
 & \hspace{0.5cm} \frac{\mathrm I_{n+p-1}\left(\frac{2\sqrt{\vecu\cdot \vecz}}{\hbar}\right)\mathrm I_{n+p-1}\left(\frac{2\sqrt{\vecz\cdot \vecu}}{\hbar}\right)}{\mathrm I_{n+p-1}\left(\frac{2|\vecz|}{\hbar}\right)} |\vecu|^{p} \mathrm K_{p}\left(\frac{2|\vecu|}{\hbar}\right) \mathrm d\vecu \mathrm d\overline\vecu\;,\label{expression star product}
\end{align}
where the functions $f_j(\vecz,\vecw)$, $j=1,2$, which are the analytic continuation of $f_j(\vecz)$  to  $\mathbb C^n\times \mathbb C^n$ (see Eq. (\ref{extended covariant symbol})).

\begin{theorem} Let $n\ge 2$, $p>-n$. The product $ *_p$ satisfies
\begin{enumerate}
\item[a) ] $f *_p 1=1*_p f= f$, for all $f \in \A$,
\item[b) ] is associative, and
\item[c) ] for $f_1, f_2 \in \A$ we have the following asymptotic expression when $\hbar \to 0$
\begin{align}
f_1 *_p f_2(\vecz) &= f_1(\vecz) f_2(\vecz) + \hbar\Bigg[g^{\bar j i}\left[\partial_{i} f_1(\vecz,\vecw)+\partial_{\bar j} f_2(\vecw,\vecz)\right]_{\vecw=\vecz} \nonumber\\
& \hspace{0.5cm}\left.+f_1(\vecz)f_2(\vecz)\left(-\frac{1}{4|\vecz|}(n-1)(n-1+2p)+\frac{R}{2}\right.\right.\nonumber\\
&\hspace{0.5cm}\left.\left.+g^{\bar j i}(\vecz)\frac{p-n+\frac{1}{2}}{2|\vecz|^4}(z_j-|\vecz|^2\delta_{ij})\right) \right]+ \mathrm O(\hbar^2)\;, \label{asymptotic expansion star product}
\end{align}
where $R$ and $g^{i\bar j}$ was defined in the theorem \ref{theorem asymptotic expansion}.
\end{enumerate}
\end{theorem}

\begin{proof}
a) Let $f \in \A$, then $f(\vecz)=\B(A)(\vecz)$ with $A\in \mathbf B(\mathcal O)$. From Eq. (\ref{extended covariant symbol}) and proposition \ref{supercomplete} 
\begin{align*}
f *_p 1 (\vecz) & = \int\limits_{\mathbb C^n} f(\vecw,\vecz) \frac{|\langle \K(\cdot,\vecz),\K(\cdot,\vecw)\rangle _{_{\Sn}}|^2}{\langle \K(\cdot,\vecz),\K(\cdot,\vecz)\rangle _{_{\Sn}}} \dm(\vecw)\\
& = \int\limits_{\mathbb C^n} \langle A\K(\cdot,\vecw),\K(\cdot,\vecz)\rangle _{_{\Sn}} \frac{\langle \K(\cdot,\vecz),\K(\cdot,\vecw)\rangle _{_{\Sn}}}{\langle \K(\cdot,\vecz),\K(\cdot,\vecz)\rangle_{_{\Sn}}} \dm(\vecw)\\
& = f(\vecz)\;.
\end{align*}
Analogously $1 *_p f=f$.

b) The associativity follows from the fact that the composition in the algebra of all bounded linear operator on $\mathcal O$ is associative.

c) The asymtotic expression given in the Eq. (\ref{asymptotic expansion star product}) is a direct consequence from the integral expression of star product (see Eq. (\ref{expression star product})), the Eq. (\ref{equality bessel = hypergeometric}) and theorem \ref{theorem asymptotic expansion}.
\end{proof}

\section{Appendix A. On the inner product on complex sphere}
In this appendix we study the inner product of functions on $\Sn$ from the form $\vartheta_{\vecz}^{\bk,s}(\vecx)=\vecx^\bk (\vecx\cdot \vecz)^s$ with $\vecz\in\mathbb C^n$, $\vecx \in \Sn$, $\bk \in \mathbb Z_+^n$ and $s\in \mathbb Z_+$. We introduce the following notation, for multi-indices $\bk,\bm \in \mathbb Z_+^n$ we will write $\bk \ge \bm$ if and only if $k_j\ge m_j$ for all $j=1,\ldots,n$.
\begin{lemma}\label{product in Sn}
 Let 
 \begin{equation*}
  F:=\langle \vartheta_{\vecz}^{\bk,s},\vartheta_{\vecz}^{\bm,\ell}\rangle_{\Sn}=\int_{\Sn} 
  \vecx^\bk \overline \vecx^\bm (\vecx \cdot \vecz)^s(\vecw \cdot \vecx)^\ell \ds(\vecx)\;,
 \end{equation*}
 with $\ell,s \in \mathbb Z_+$, $\bk,\bm \in \mathbb Z_+^n$, and $\vecz,\vecw \in \mathbb C^n\backslash\{0\}$.
 Then $F=0$ if $|\bk|+s\ne|\bm|+\ell$, furthermore
 \begin{enumerate}
  \item[ a)] If $|\bk|+s=|\bm|+\ell$ and $\bm \ge \bk$, then
  \begin{equation*}
   F=\frac{(n-1)!(|\bm|-|\bk|+\ell)!}{(n-1+|\bm|+\ell)!} \frac{\overline\vecz^\bm}{\overline\vecz^\bk} \; \sum_{|\bbeta|=\ell}\frac{\ell!}{\bbeta!} \frac{(\bm+\bbeta)!}{(\bm-\bk+\bbeta)!} \overline \vecz^{\bbeta} \vecw^{\bbeta}\;.
  \end{equation*}
  
  \item[ b)] If $|\bk|+s=|\bm|+\ell$, and $\bk \ge \bm$, then
  \begin{equation*}
   F=\frac{(n-1)!(s+|\bk|-|\bm|)!}{(|\bk|+s+n-1)!} \frac{\vecw^\bk}{\vecw^\bm} \; \sum_{|\balpha|=s}\frac{s!}{\balpha!} \frac{(\bk+\balpha)!}{(\bk-\bm+\balpha)!} \overline \vecz^{\balpha} \vecw^{\balpha}\;.
  \end{equation*}
 \end{enumerate}
\end{lemma}
\begin{proof}
We note that the function $\vartheta_{\vecz}^{\bk,s}(\vecx)$ is harmonic and homogeneous of order $|\bk|+s$, then $J=0$ if  $|\bk|+s\ne|\bm|+\ell$. We suppose now that $|\bk|+s=|\bm|+\ell$, and $\bm \ge \bk$, then
\begin{align*}
F & = \sum_{|\balpha|=s}\sum_{|\bbeta|=\ell}\frac{s!}{\balpha!}\frac{\ell}{\bbeta!} \overline \vecz^{\balpha}\vecw^{\bbeta} \int_{\Sn}  \vecx^{\bk+\balpha} \overline \vecx^{\bm+\bbeta} \ds(\vecx)\\
& = \sum_{|\bbeta|=\ell}\frac{(|\bm|-|\bk|+\ell)!}{(\bm-\bk+\bbeta)!} \frac{\ell!}{\bbeta!} \overline \vecz^{\bm-\bk+\bbeta}\vecw^{\bbeta}\frac{(n-1)!(\bm+\bbeta)!}{(n-1+|\bm|+\ell)!}\\
& = \overline \vecz^{\bm-\bk}\frac{(n-1)!(|\bm|-|\bk|+\ell)!}{(n-1+|\bm|+\ell)!}\sum_{|\bbeta|=\ell}\frac{(\bm+\bbeta)!}{(\bm-\bk+\bbeta)!} \frac{\ell!}{\bbeta!} \overline \vecz^{\bbeta}\vecw^{\bbeta}\;,
\end{align*}
where we use $\displaystyle\int_{\Sn}|\vecx^{\balpha}|^2 \ds(\vecx)=\frac{(n-1)!\balpha!}{(n-1+|\balpha|)!}$ (see Ref \cite{R08} for details). Similarly we get the other formula.
\end{proof}

\section{Appendix B. Asymptotic of the Integral kernel of $\mathbf U_{n,-1}$} \label{AppendixB}

In the particular case when $p=-1$, we can obtain the asymptotic expresion of the coherent states $\K$. First, based on the definition of $\K$  (see Eqs. (\ref{coherent states}) and  (\ref{constants coherent states})) let us define the function $g : \mathbb C\to \mathbb C$ by
\begin{displaymath}
g(z)= \sum_{\ell=0}^\infty \frac{\sqrt{a\ell+1}}{\ell !} \;z^\ell\;,\hspace{1cm} a=\frac{1}{n-1}\,.
\end{displaymath}

Note that the coherent states $\mathbf U_{n,-1}$ are equal to the function $g$ evaluated at $(\vecx\cdot \vecz)/\hbar$. In this appendix we obtain the main asymptotic term for the function $g$.
\begin{lemma}\label{asymptotic p=1}  
For $\Re(z)>0$ and $|\Im(z)| \le C \Re(z)$ with $C$ a positive constant and $\Re(z)\to +\infty$, $g$ has the following asymptotic expansion:
\begin{equation}
g(z)=\sqrt a z^{1/2}\mathrm{exp}(z)\left[ 1+\frac{a_1}{z}+\frac{a_2}{z^2}+\mathrm O(z^{-3})\right]\;,
\end{equation}
with $a_1,a_2$ some constants.
\end{lemma}
\begin{proof}
This follows from Lemma 10.1, which appears in Ref. \cite{D-V09}.
\end{proof}

Using Lemma \ref{asymptotic p=1} with $z = \vecx\cdot\vecz/\hbar$ we obtain the following asymptotic expansion:
\begin{prop}\label{asymptotic coherent states p=1}
Let $\vecz \in \mathbb C^n-\{0\}$. Then for $\hbar \to 0$ and $|\Im(\vecx \cdot\vecz)|\le C\Re(\vecx\cdot\vecz)$, with $C$ a positive constant, we have
\begin{equation*}
\K(\vecx,\vecz)=\left[\frac{\vecx\cdot\vecz}{\hbar(n-1)}\right]^{\frac{1}{2}} \mathrm{exp}\left(\frac{\vecx\cdot\vecz}{\hbar}\right)\left[ 1+\frac{a_1}{\vecx\cdot\vecz} \hbar+\mathrm O(\hbar^2)\right]\;.
\end{equation*}
\end{prop}

\newpage
\end{document}